\documentclass[11pt,a4paper]{article}
\usepackage{amsmath,amssymb,amsthm}
\usepackage[utf8]{inputenc}
\usepackage[T1]{fontenc}
\usepackage{geometry}
\usepackage[colorlinks=true,linkcolor=blue,citecolor=blue,urlcolor=blue]{hyperref}

\newtheorem{theorem}{Theorem}[section]
\newtheorem{lemma}[theorem]{Lemma}
\newtheorem{corollary}[theorem]{Corollary}
\newtheorem{proposition}[theorem]{Proposition}
\theoremstyle{definition}
\newtheorem{remark}[theorem]{Remark}

\newcommand{\od}{\overrightarrow{d}}
\newcommand{\diam}{\operatorname{diam}}
\newcommand{\rad}{\operatorname{rad}}
\newcommand{\dist}{\operatorname{dist}}
\newcommand{\ecc}{\operatorname{ecc}}

\title{The asymptotic maximum oriented diameter of graphs}
\author{
  Jiangdong Ai\thanks{School of Mathematical Sciences and LPMC,
Nankai University. {\tt jd@nankai.edu.cn}.
 Funded by the National Natural Science Foundation of China (No.\,12522117)
and the Fundamental and Interdisciplinary Disciplines Breakthrough Plan
of the Ministry of Education of China (JYB2025XDXM207).}
  \hspace{2mm}
  Yaokun Feng\thanks{Center for Combinatorics and LPMC, Nankai University, and Department of Applied Mathematics, The Hong Kong Polytechnic University, Hong
Kong, China. {\tt 1031679547@qq.com}.}
 \hspace{2mm} Hui Lei\thanks{School of Statistics and Data Science, LPMC and KLMDASR, Nankai University. {\tt hlei@nankai.edu.cn}. Funded by the National Natural Science Foundation of China (Nos.\,12371351, 12431013), the Natural Science Foundation of Tianjin (24JCYBJC01670),
and the Fundamental Research Funds for the Central Universities, Nankai University (63263098).}
\hspace{2mm}
  Zijian Ren\thanks{School of Mathematical Sciences and LPMC,
    Nankai University. {\tt 2258544653@qq.com}.}}
\date{}

\begin{document}
\maketitle

\begin{abstract}
The oriented diameter of a connected bridgeless graph is the minimum
diameter of a strong orientation. Let $f(d)$ be the maximum oriented
diameter among all such graphs of diameter $d$.
In 1978, Chv\'atal and Thomassen proved $f(d)\le2d^2+2d$ and
constructed graphs showing that any quadratic upper bound on $f(d)$
must have leading coefficient at least $1/2$.
We prove that $f(d)\le \tfrac12d^2+7d$ for every integer $d\ge1$.
This matches the leading coefficient of their lower bound and establishes
$f(d)=\tfrac12d^2+O(d)$, determining the optimal quadratic coefficient.
Our proof gives a polynomial-time algorithm that constructs a strong
orientation satisfying the stated bound.
\end{abstract}
\section{Introduction}

Throughout, all graphs are finite, simple and connected unless stated
otherwise. A graph is \emph{bridgeless} if it has no edge whose deletion
disconnects it. An \emph{orientation} assigns one direction to each edge,
and an orientation is \emph{strong} if every vertex can reach every other
vertex by a directed path. Robbins~\cite{R39} proved that a connected
graph has a strong orientation if and only if it is bridgeless.

For vertices $x,y$ of a graph $G$, let $\dist_G(x,y)$ denote their
distance. For a nonempty set $A\subseteq V(G)$, define
$\dist_G(x,A)=\min_{a\in A}\dist_G(x,a)$.
The eccentricity, radius and diameter are
\[
 \ecc_G(x)=\max_{y\in V(G)}\dist_G(x,y),\qquad
 \rad(G)=\min_{x\in V(G)}\ecc_G(x),\qquad
 \diam(G)=\max_{x\in V(G)}\ecc_G(x).
\]
A vertex attaining the radius is a \emph{center}.

A subgraph $H$ of $G$ is \emph{$t$-step dominating} if
$\dist_G(x,V(H))\le t$ for every $x\in V(G)$.
For a strong orientation $\vec G$, distances are measured along directed
paths, and
\[
 \diam(\vec G)=\max_{x,y\in V(G)}\dist_{\vec G}(x,y).
\]
The \emph{oriented diameter} of a bridgeless graph is
\[
 \od(G)=\min\{\diam(\vec G):\vec G
                 \text{ is a strong orientation of }G\}.
\]
For each integer $d\ge1$, define
\[
 f(d)=\max\{\od(G):G\text{ is bridgeless and }\diam(G)=d\}.
\]
We regard the one-vertex graph as strongly oriented with diameter zero.

Chv\'atal and Thomassen~\cite{CT78} initiated the study of this parameter.
They proved that $f(d)\le2d^2+2d$, and constructed cactus graphs of
even diameter $d$ with oriented diameter $\tfrac12d^2+d$.
These bounds left the optimal quadratic coefficient between $\tfrac12$ and $2$.
Babu, Benson, Rajendraprasad and Vaka~\cite{BBRV21} subsequently obtained
$f(d)\le1.373d^2+O(d)$.
Their proof uses an oriented core: a strongly oriented subgraph with
small directed diameter and a controlled domination distance.
More recently, Ai, Lei and Ren~\cite{ALR26} proved
\[
 f(d)\le d^2+\left\lceil\tfrac{11}{2}d\right\rceil-5
 \qquad(d\ge3).
\]
We reduce the quadratic coefficient to $\tfrac12$, matching the
lower-bound coefficient.

For small diameters, Chv\'atal and Thomassen~\cite{CT78} proved
$f(2)=6$, while the work of Kwok, Liu and West~\cite{KLW10} and
Wang and Chen~\cite{WC22} established $f(3)=9$.
For $d=4$, Lin, Wang and You~\cite{LWY26} obtained $f(4)\le18$,
and Feng, Lei, Lian and Ren~\cite{FLLR26} subsequently improved this
to $f(4)\le16$.

For an edge $e$ of a bridgeless graph $G$, let $\ell_G(e)$ be the length
of a shortest cycle containing $e$, and write
\[
 \eta(G)=\max_{e\in E(G)}\ell_G(e).
\]
Huang, Li, Li and Sun~\cite{HLLS14} used this parameter to prove
\[
 \od(G)\le2\sum_{i=1}^{\rho}\min\{2i,\eta(G)-1\},
 \qquad \rho=\rad(G).
\]
Our main result retains these extension terms, but reduces the number
of rounds to $\lfloor d/2\rfloor$ by first constructing an oriented
core of diameter at most $6\rho$.

\begin{theorem}\label{thm:main}
Let $G$ be a bridgeless graph with diameter $d\ge1$, radius $\rho$,
and $\eta=\eta(G)$. Let $R=\lfloor d/2\rfloor$.
Then $G$ has a strong orientation $\vec G$ such that
\begin{equation}\label{eq:main}
 \begin{aligned}
 \diam(\vec G)
 &\le 6\rho+2\sum_{i=1}^{R}\min\{2i,\eta-1\}\\
 &\le 2R(R+1)+6\rho
 \le \tfrac12d^2+7d.
 \end{aligned}
\end{equation}
In particular, $f(d)\le\tfrac12d^2+7d$ for every $d\ge1$.
\end{theorem}

The bound $\tfrac12d^2+7d$ in Theorem~\ref{thm:main} strictly improves
the bound of Ai, Lei and Ren~\cite{ALR26} for every integer $d\ge5$.
It need not improve the radius--$\eta$ bound above for every choice
of the parameters.
The new part of the argument is the following core theorem.

\begin{theorem}\label{thm:core}
Every bridgeless graph $G$ of diameter $d\ge1$ and radius $\rho$ contains
an $\lfloor d/2\rfloor$-step dominating subgraph $K$ with a strong
orientation $\vec K$ satisfying
$\diam(\vec K)\le6\rho$. We refer to any such strongly oriented subgraph $\vec K$
as an \emph{oriented core} of $G$.
\end{theorem}

The domination distance in Theorem~\ref{thm:core} is responsible for the
coefficient $\tfrac12$. Once the core is oriented, a shortest-ear
construction extends its orientation at a cost of at most $2R(R+1)$.
We give a complete proof of this extension step, following the
shortest-ear approach of Chv\'atal and Thomassen~\cite{CT78} and
Huang, Li, Li and Sun~\cite{HLLS14}, with the orientation of the core
preserved throughout.

To obtain the core, we use the fundamental cycles of a
shortest-path spanning tree of $G$. We replace each node of this tree
by a path to obtain an indexing tree of maximum degree at most three,
without changing $G$. The corresponding fundamental-cycle bags form
a tree decomposition of $G$. A Helly argument selects a bag which is
$R$-step dominating. The at most three components obtained by deleting
its indexing node provide three labels for the original vertices other
than the root. These labels divide the relevant fundamental cycles into
two groups, each of which can be oriented consistently in $G$.
An asymmetric combination preserves short return paths to the root,
giving distances at most $2\rho$ to the root and $4\rho$ from it,
and hence the bound $6\rho$.

Section~\ref{sec:extension} establishes the extension lemma.
Section~\ref{sec:core} proves Theorem~\ref{thm:core} and completes the
proof of Theorem~\ref{thm:main}.
Section~\ref{sec:lower} gives a self-contained cactus construction
showing that the leading coefficient is best possible.

\section{Extending a strong orientation}\label{sec:extension}

We give the extension bound, including the compatibility argument
for overlapping ears.

\begin{lemma}[Orientation-preserving extension]\label{lem:ear-extension}
Let $G$ be a finite connected bridgeless graph with at least two vertices,
and let $H$ be a nonempty
$t$-step dominating subgraph of $G$, where $t$ is a nonnegative integer. Put $\eta=\eta(G)$.
Every strong orientation $\vec H$ of $H$ extends to a strong orientation $\vec G$
of $G$ satisfying
\[
 \diam(\vec G)
 \le \diam(\vec H)
       +2\sum_{i=1}^{t}\min\{2i,\eta-1\}
 \le \diam(\vec H)+2t(t+1).
\]
A single vertex is allowed as a strongly oriented subgraph of diameter zero.
\end{lemma}
\begin{proof}
For a nonempty set $A\subsetneq V(G)$, an $A$-ear is a path with its
endpoints in $A$ and all internal vertices outside $A$; we also allow a
cycle meeting $A$ in exactly one vertex. For an edge $e=av$ with $a\in A$
and $v\notin A$, an $(A,e)$-ear has first edge $av$ when traversed from
$a$. Shortest always refers to this prescribed first undirected edge.
The eventual consistent orientation of an ear may follow either traversal.

We first record the distance profile of a shortest $(A,e)$-ear
$P=x_0x_1\cdots x_\ell$, where $x_0=a$ and $x_1=v$.
Such an ear exists: follow a cycle containing $e$ from $a$ along $e$
until its first return to $A$. Moreover,
$P[v,x_\ell]$ is a shortest $v$--$A$ path in $G-e$, since prepending
$av$ to any shortest such path gives an $(A,e)$-ear.
For $1\le j<\ell$, every $x_j$--$A$ path avoiding $e$ has length at
least $\ell-j$, since adjoining it to $P[v,x_j]$ and deleting
closed subwalks would otherwise give a shorter $v$--$A$ path in
$G-e$. A shortest
$x_j$--$A$ path using $e$ must traverse $v a$ as its last edge.
Its preceding part has length at least $j-1$, because $P[v,x_j]$ is
geodesic in $G-e$. Following either side of $P$ gives the reverse
inequality, and hence
\begin{equation}\label{eq:ear-profile}
 \dist_G(x_j,A)=\min\{j,\ell-j\}.
\end{equation}
Consequently, if $A$ is $s$-step dominating, a middle vertex gives
$\ell\le2s+1$. The cycle construction above also gives
$\ell\le\ell_G(e)\le\eta$.

Now fix an $s$-step dominating set $A$, with $s\ge1$, carrying the
current strong orientation. We enlarge its vertex set to contain every
neighbour of $A$, while keeping every existing arc. Throughout this
round, $A$ remains fixed. Record every ear used in the round, including
its consistent orientation. Call a vertex covered if it belongs to $A$
or to an already recorded ear.

If an uncovered neighbour $v$ of $A$ remains, choose $e=av$ and a
shortest $(A,e)$-ear $P$. If all internal vertices of $P$ are uncovered,
orient it consistently and record it. Otherwise let $w$ be the first
covered internal vertex of $P$, starting from $a$. Choose an earlier
recorded ear $Q$ containing $w$. A shorter side of $Q$ from $w$ to $A$
has length $\dist_G(w,A)$ by~\eqref{eq:ear-profile}.
Replace the part of $P$ after $w$ by that side of $Q$, obtaining $P'$.

All internal vertices of $P[a,w]$ are uncovered, whereas every vertex
of $Q$ is covered. Thus the two pieces meet only at $w$ and possibly
at their endpoints in $A$. They form a valid $(A,e)$-ear, open or closed.
Its length is no greater than that of $P$, so the minimality of $P$
makes $P'$ a shortest $(A,e)$-ear as well. The chosen side of $Q$ is
already directed in one direction. Orient the new prefix $P[a,w]$
in the matching direction to make all of $P'$ a directed ear, and
record $P'$. No old edge is reoriented, since every prefix edge has
an uncovered endpoint. The newly recorded ear contains $v$, so this
procedure terminates after finitely many additions.

Every newly covered vertex lies on a recorded directed ear of length
at most $\min\{2s+1,\eta\}$. It can therefore reach $A$ and be reached
from $A$ within
$b(s)=\min\{2s,\eta-1\}$ steps in the enlarged orientation.
All these paths coexist, and strong connectivity is preserved.
Since all neighbours of $A$ have been covered, the enlarged vertex set
is $(s-1)$-step dominating: a shortest path to $A$ meets a neighbour
of $A$ one step before reaching $A$.

Starting with $A=V(H)$, perform these rounds with
$s=t,t-1,\ldots,1$, stopping earlier if every vertex has been covered.
Every vertex can then reach $V(H)$ and be reached from $V(H)$ within
$\sum_{i=1}^{t}b(i)$ steps. Routing an ordered pair through the
unchanged orientation $\vec H$ proves the asserted bound. Orient all
remaining edges arbitrarily. The case $t=0$ only requires this final step.
\end{proof}

\section{Construction of the oriented core}\label{sec:core}
\subsection{Central bags and fundamental cycles}

Recall that a \emph{tree decomposition} of a graph $X$ consists of a tree
$T$ and sets $B_t\subseteq V(X)$, called \emph{bags}, such that every
vertex lies in a bag, the ends of every edge lie together in a bag,
and for each $x\in V(X)$ the set $\{t:x\in B_t\}$ induces a connected
subtree of $T$.

\begin{lemma}[A central bag]\label{lem:central}
Let $X$ be a connected graph of diameter $d$.
Every tree decomposition of $X$ has a bag $B$ such that
\[
 \dist_X(x,B)\le\lfloor d/2\rfloor
 \qquad\text{for every }x\in V(X).
\]
\end{lemma}
\begin{proof}
Put $R=\lfloor d/2\rfloor$. For each vertex $x$, let
$A_x=\{y:\dist_X(x,y)\le R\}$.
This set is connected, since the initial segments of a shortest path
from $x$ stay in $A_x$. The nodes of the decomposition tree whose bags
meet $A_x$ therefore form a nonempty subtree $T_x$.
Indeed, the bags containing a fixed vertex form a subtree, and the
supports of adjacent vertices meet in an edge-containing bag.

For any $x,y$, a shortest $x$--$y$ path has integer length at most
$d\le2R+1$. Its portions within distance $R$ of its ends either meet
or are joined by an edge. Thus $A_x,A_y$ intersect or have
an edge between them. In both cases $T_x\cap T_y\ne\varnothing$.
The Helly property of subtrees of a tree gives a node belonging to
every $T_x$. Its bag meets every $A_x$, as required.
\end{proof}

\begin{lemma}[Fundamental-cycle bags]\label{lem:cycle-bags}
Let $G$ be a connected bridgeless graph with at least two vertices,
and let $T$ be a spanning tree of $G$.
For each $x\in V(T)$, take a path $Q_x$ on $\deg_T(x)$ nodes, with
these paths pairwise disjoint. Assign the tree edges incident with
$x$ bijectively to the nodes of $Q_x$. For each $xy\in E(T)$, join
the nodes of $Q_x$ and $Q_y$ assigned to $xy$, and denote the resulting
graph by $\mathcal T$. Choose a representative $r_x\in V(Q_x)$ for each $x$.
For each edge $e=ab\in E(G)\setminus E(T)$, write
$P_e=T[a,b]$, $C_e=P_e+e$ and $\widehat P_e=\mathcal T[r_a,r_b]$.
For $t\in V(Q_v)$, let $H_t$ consist of the vertex $v$ and all cycles
$C_e$ with $t\in V(\widehat P_e)$, and let $B_t=V(H_t)$.
Then $\mathcal T$ is a tree of maximum degree at most three, and
$(\mathcal T,\{B_t:t\in V(\mathcal T)\})$ is a tree decomposition of $G$.
Moreover, $H_t$ is nontrivial, connected and bridgeless,
$T_t=H_t\cap T$ is a spanning tree of $H_t$, and every fundamental
cycle of $H_t$ relative to $T_t$ contains $v$.
If the cycles $C_e$ have length at most $L$, the same holds for these
fundamental cycles of $H_t$.
\end{lemma}
\begin{proof}
Contracting the paths $Q_x$ gives $T$, so $\mathcal T$ is a tree.
Each node has at most two neighbours on its path $Q_x$ and exactly
one neighbour outside that path. Thus its degree is at most three.
The contraction maps $\widehat P_e$ onto $P_e$. In particular,
$\widehat P_e$ meets $Q_x$ if and only if $x\in V(P_e)$.

Every vertex $x$ lies in every bag indexed by a node of $Q_x$.
The ends of a non-tree edge $e$ lie together in each bag indexed
by $\widehat P_e$.
Every tree edge $xy$ belongs to some $C_e$, since otherwise its
fundamental cut would make it a bridge of $G$.
The corresponding $\widehat P_e$ uses the edge joining $Q_x$ to $Q_y$,
so a bag indexed by either end of that edge contains both $x$ and $y$.
For a fixed vertex $x$ of $G$, its support in $\mathcal T$ is exactly
\[
 \{t:x\in B_t\}
   =V(Q_x)\cup
     \bigcup_{\substack{e\in E(G)\setminus E(T)\\x\in V(P_e)}}
        V(\widehat P_e).
\]
Each path in the union meets $Q_x$. Hence the support is connected,
which proves the tree-decomposition axioms.

For each $t\in V(Q_v)$, let $vy$ be the tree edge assigned to $t$.
This edge belongs to some $C_e$, and the corresponding $\widehat P_e$
uses the edge joining $Q_v$ to $Q_y$ at $t$. Hence $H_t$ contains
$C_e$ and is nontrivial.
If $t\in V(Q_v)$ and $t\in V(\widehat P_e)$, then $v\in V(P_e)$.
Thus all cycles defining $H_t$ contain $v$. Their union with $v$ is
connected and bridgeless, and the union of their tree paths with $v$
is a subtree spanning $H_t$. Each non-tree edge of $H_t$ is one of
the selected edges and retains its original tree path, so its
fundamental cycle is exactly its original $C_e$.
\end{proof}

\subsection{Asymmetric grouping of fundamental cycles}

The next lemma allows different branches to have the same label.
It does not require a bound on the number of branches or cycles.

\begin{lemma}[Asymmetric grouping]\label{lem:asymmetric}
Let $H$ be a connected bridgeless graph with at least two vertices,
and let $U$ be a spanning tree rooted at $z$.
Suppose that every non-tree fundamental cycle contains $z$ and has
length at most $L$. Suppose also that there is a map
$\lambda:V(H)\setminus\{z\}\to\{1,2,3\}$ which is constant on each
component of $U-z$ and satisfies $\lambda(a)\ne\lambda(b)$ for every
$ab\in E(H)\setminus E(U)$ with $a,b\ne z$.
Then $H$ has a strongly oriented spanning subgraph $\vec K$ satisfying
\[
 \dist_{\vec K}(x,z)\le L-1,\qquad
 \dist_{\vec K}(z,x)\le2(L-1)
 \quad\text{for every }x\in V(H).
\]
In particular, $\diam(\vec K)\le3(L-1)$.
\end{lemma}
\begin{proof}
Call the components of $U-z$ the \emph{branches}, and let $A,B,C$
be the unions of the branches with labels $1,2,3$, respectively.
Each tree edge joining a branch to $z$ is treated as part of that branch.
A non-tree edge not incident with $z$ joins two different label classes.

Let $J_1$ be the union of all fundamental cycles except those whose
non-tree edge joins $B$ to $C$, and let $J_2$ be the union of these
remaining cycles. Include $z$ in each group, even when its family
of cycles is empty.
In $J_1$, orient the tree edges in $A$ away from $z$ and those in
$B,C$ towards $z$. Direct non-tree edges between label classes from
$A$ to $B$ or $C$. Direct a non-tree edge incident with $z$ towards
$z$ if its other endpoint is in $A$, and away from $z$ otherwise.
In $J_2$, orient the tree edges in $B$ away from $z$, those in $C$
towards $z$, and the non-tree edges from $B$ to $C$.
Every selected fundamental cycle is directed consistently within its
group. Every tree edge of $H$ belongs to a fundamental cycle, since
$H$ is bridgeless, so $J_1\cup J_2=H$.
Each nonroot vertex of either group lies on a selected directed cycle
of length at most $L$, and hence has distances at most $L-1$ in both
directions between itself and $z$ within that group.

Put $S=V(J_1)$. The tree part $J_1\cap U$ is a subtree containing
$z$ and spanning $S$. Thus the entire tree path to $z$ of a vertex
in $S$ lies in $J_1$; in particular, no descendant of a vertex
outside $S$ belongs to $S$.
Keep all arcs of $\vec J_1$, and add every arc of $\vec J_2$ having
at least one endpoint outside $S$. No added edge was already oriented,
so this defines an oriented spanning subgraph $\vec K$ of $H$.

Let $x\notin S$, and choose a selected directed cycle of $J_2$
through $x$. If $x\in B$, its directed segment to $z$ first goes
away from $z$ in $B$, crosses to $C$, and then follows $C$ towards $z$.
All vertices of the first part are descendants of $x$ and hence
outside $S$, so its edges and the crossing edge are retained.
On the $C$ part, an edge with an endpoint outside $S$ is added,
whereas an edge with both endpoints in $S$ already belongs to $J_1$
and has the same direction towards $z$.
Thus the whole return segment survives. If $x\in C$, its return
segment is its tree path to $z$, and the same argument applies.
In either case, $\dist_{\vec K}(x,z)\le L-1$.

On the directed $z$--$x$ segment of the chosen cycle, let $w$ be
the last vertex in $S$. Every edge of the suffix from $w$ to $x$
has an endpoint outside $S$, so this suffix survives and has length
at most $L-1$. Reach $w$ from $z$ inside $\vec J_1$ within $L-1$
steps, using a length-zero path if $w=z$, and then follow the suffix.
Hence $\dist_{\vec K}(z,x)\le2(L-1)$.
Vertices in $S$ retain their original paths in $\vec J_1$.
These bounds prove strong connectivity and the stated diameter bound.
\end{proof}

\subsection{Choosing and orienting the core}

\begin{proof}[Proof of Theorem~\ref{thm:core}]
Write $d=\diam(G)$, $\rho=\rad(G)$ and $R=\lfloor d/2\rfloor$.
Choose a center $c$ of $G$ and a breadth-first-search spanning tree
$T$ rooted at $c$. Its depth is at most $\rho$, so every fundamental
cycle $C_e$ has length at most $L=2\rho+1$.
Apply Lemma~\ref{lem:cycle-bags} to $G,T$, obtaining an indexing tree
$\mathcal T$, paths $Q_x$, representatives $r_x$ and bags $B_t=V(H_t)$.
Only the indexing tree is refined; the graph $G$ and its distances
are unchanged.
Lemma~\ref{lem:central} supplies a bag $B_t$ which is $R$-step dominating
in $G$. Let $t\in V(Q_v)$.

Assign distinct labels from $\{1,2,3\}$ to the at most three
components of $\mathcal T-t$.
For each $x\ne v$, give $x$ the label of the component containing $r_x$.
This defines a map $\lambda:V(G)\setminus\{v\}\to\{1,2,3\}$.
For every component $A$ of $T-v$, the paths $Q_x$ with $x\in V(A)$
and the edges corresponding to $E(A)$ form a connected subtree of
$\mathcal T$ avoiding $t$.
Hence $\lambda$ is constant on $V(A)$, and therefore on each component
of $T_t-v$, where $T_t=H_t\cap T$.

For each non-tree edge $e=ab$ of $H_t$ with $a,b\ne v$, the path
$\widehat P_e=\mathcal T[r_a,r_b]$ contains $t$.
Neither endpoint is $t$, so $r_a,r_b$ lie in different components
of $\mathcal T-t$, and therefore $\lambda(a)\ne\lambda(b)$.
All fundamental cycles of $H_t$ relative to $T_t$ contain $v$ and
have length at most $L$ by Lemma~\ref{lem:cycle-bags}.
Lemma~\ref{lem:asymmetric} therefore gives a strongly oriented
spanning subgraph $\vec K$ of $H_t$, within the original graph $G$,
such that
\begin{equation}\label{eq:group-root}
 \dist_{\vec K}(x,v)\le L-1=2\rho,\qquad
 \dist_{\vec K}(v,x)\le2(L-1)=4\rho
 \quad(x\in V(K)).
\end{equation}
Consequently,
\[
 V(K)=B_t,\qquad \diam(\vec K)\le3(L-1)=6\rho.
\]
Since $B_t$ is $R$-step dominating in $G$, so is $K$.
This completes the proof.
\end{proof}

\subsection{Completion of the main proof}

\begin{proof}[Proof of Theorem~\ref{thm:main}]
Take the oriented core $\vec K$ from Theorem~\ref{thm:core}.
Apply Lemma~\ref{lem:ear-extension} in the original graph $G$, with
$t=R=\lfloor d/2\rfloor$. It extends $\vec K$ to a strong orientation
$\vec G$ satisfying
\[
 \diam(\vec G)
 \le6\rho+2\sum_{i=1}^{R}\min\{2i,\eta-1\}
 \le6\rho+2R(R+1).
\]
Finally, $\rho\le d$ and $2R(R+1)\le d^2/2+d$ give~\eqref{eq:main}.
\end{proof}

\begin{remark}\label{rem:algorithm}
The proof gives a polynomial-time construction.
Distances and a center can be found by breadth-first search in $G$.
If $G$ has $n$ vertices, the indexing tree $\mathcal T$ has
$\sum_x\deg_T(x)=2n-2$ nodes and $2n-3$ edges.
The spanning tree, its fundamental cycles, the paths $\widehat P_e$
and all bags can be enumerated in polynomial time.
Testing the bags finds a central bag without requiring an algorithmic
form of the Helly theorem. The component labels and the asymmetric
two-group construction are explicit. Finally, a shortest prescribed-leg ear is obtained by
breadth-first search after deleting its prescribed edge; each ear
addition covers a new vertex. All steps therefore take polynomial time.
\end{remark}

\section{Sharpness of the leading coefficient}\label{sec:lower}

For completeness, we give a cactus construction of the classical type
used by Chv\'atal and Thomassen~\cite{CT78}. The odd-diameter case of
the classical construction was corrected by Kwok, Liu and
West~\cite{KLW10}; see also Huang et al.~\cite{HLLS14}.
We include both parity cases below.
A \emph{cactus} is a connected graph in which every block is an edge
or a cycle. The cacti below have only cycle blocks and are bridgeless.

Let $B_0$ consist of a single root. For $r\ge1$, form a rooted graph
$B_r$ from a cycle of length $2r+1$ with root $z$ by attaching a copy
of $B_{r-1}$ at each of the two neighbours of $z$, identifying the
root of each copy with that neighbour.
Put $a_r=r(r+1)$.

\begin{lemma}\label{lem:cactus-root}
The root of $B_r$ has eccentricity $r$.
In every strong orientation of $B_r$, its maximum directed distance
to a vertex and its maximum directed distance from a vertex are
both equal to $a_r$.
\end{lemma}
\begin{proof}
The assertion is immediate for $r=0$.
Every vertex on the new cycle is within distance $r$ of its root,
and every vertex in an attached copy is within distance
$1+(r-1)=r$. Equality follows from the induction hypothesis in an
attached copy, so the undirected eccentricity is $r$.

Every component outside a cycle block of a cactus attaches to that
block at a single vertex. An excursion outside the block therefore
returns to its departure vertex and can be deleted from a directed
walk. The same argument applies to an attached copy, which meets the
rest of the graph only at its root.
Thus a strong orientation restricts to a strong orientation of each
cycle block and attached copy, and their internal directed distances
are unchanged. In particular, every cycle block is oriented cyclically.

Let $p,q$ be the two neighbours of the root on the new cycle, labelled
so that $p\to z\to q$.
The directed distance from $z$ to $p$, and from $q$ to $z$, is $2r$.
By induction, the copy rooted at $p$ contains a vertex at directed
distance $a_{r-1}$ from $p$, while the copy rooted at $q$ contains
a vertex at directed distance $a_{r-1}$ to $q$.
They give root distances at least
$2r+a_{r-1}=a_r$ in the two directions.
Conversely, a route from or to the root uses at most $2r$ edges of
the new cycle and at most $a_{r-1}$ edges in an attached copy.
Both maximum directed distances are therefore exactly $a_r$.
\end{proof}

\begin{proposition}\label{prop:lower}
For every integer $d\ge1$, there is a bridgeless cactus $G_d$ such that
\[
 \diam(G_d)=d,\qquad
 \od(G_d)=\left\lfloor\frac{(d+1)^2}{2}\right\rfloor.
\]
\end{proposition}
\begin{proof}
First let $d=2r$, where $r\ge1$.
Take two copies of $B_r$ and identify their roots.
Lemma~\ref{lem:cactus-root} gives diameter $2r$: vertices at distance
$r$ from the common root in different copies attain this value.
In any strong orientation, choose a vertex in the first copy whose
directed distance to the root is $a_r$, and a vertex in the second
whose directed distance from the root is $a_r$.
Every directed path between them passes through the root, so its
length is at least $2a_r$. Conversely, routing through the root
gives diameter at most $2a_r$.
Thus $\od(G_{2r})=2r(r+1)$.

Now let $d=2r+1$, where $r\ge0$.
Take a cycle $C$ of length $2r+3$ and attach copies of $B_r$
at two adjacent vertices $p,q$ of $C$.
Two vertices in different copies are at distance at most $2r+1$.
Two vertices in one copy are at distance at most $2r$;
a vertex of a copy and a vertex of $C$ are at distance at most
$r+(r+1)$; and two vertices of $C$ are at distance at most $r+1$.
Vertices at distance $r$ from $p$ and $q$ in the two copies attain
distance $2r+1$. This also covers $r=0$.

Every strong orientation directs $C$ cyclically.
In one of the two directions between $p$ and $q$, the directed
distance along $C$ is $2r+2$.
Choose vertices in the two copies whose directed distances to the
initial attachment vertex and from the terminal attachment vertex
are both $a_r$. Any directed path between them must use this
long directed segment of $C$, giving a lower bound
\[
 2a_r+2r+2=2(r+1)^2.
\]
For the reverse inequality, pairs within an attached copy can be
routed through its root in at most $2a_r$ steps.
All other pairs can be joined using at most $a_r$ steps in each
attached copy and at most $2r+2$ steps on $C$.
Thus the oriented diameter is exactly $2(r+1)^2$.
The two parity cases give the stated formula.
\end{proof}

\begin{corollary}\label{cor:asymptotic}
For every integer $d\ge1$,
\[
 \left\lfloor\frac{(d+1)^2}{2}\right\rfloor
 \le f(d)\le\tfrac12d^2+7d.
\]
Consequently,
\[
 f(d)=\tfrac12d^2+O(d)
 \qquad\text{and}\qquad
 \lim_{d\to\infty}\frac{f(d)}{d^2}=\frac12.
\]
\end{corollary}
\begin{proof}
Combine Proposition~\ref{prop:lower} with Theorem~\ref{thm:main}.
\end{proof}

The quadratic term is therefore determined. The remaining difference
between these bounds is linear in $d$. Determining the optimal linear
term, or the exact values of $f(d)$ in general, requires a further
argument.

\section*{Artificial intelligence use declaration}

OpenAI's ChatGPT was used to assist with mathematical exploration,
proof development, computational checks, and the drafting and language
editing of this manuscript. The authors take full responsibility for the results presented in this manuscript.

\end{document}